\documentclass[12pt,a4paper,oneside]{amsart}
\usepackage{amsfonts, amsmath, amssymb, amsthm, mathtools}
\usepackage[colorlinks=true,citecolor=blue]{hyperref}
\usepackage[margin=1.4in]{geometry}
\usepackage[T1]{fontenc}
\usepackage{newtxtext}
\usepackage{newtxmath}
\usepackage{graphicx}
\usepackage{verbatim}
\usepackage{tikz}
\usepackage{pgfplots}
\pgfplotsset{compat=1.18}
\definecolor{cChordal}{RGB}{31,119,180}
\definecolor{cHolmsen}{RGB}{214,39,40}
\definecolor{cLow}{RGB}{106,61,154}
\definecolor{cHigh}{RGB}{0,0,0}
\usepackage{fix-cm}
\usepackage{tikz-cd}
\usetikzlibrary{calc,matrix}

\makeatletter
\renewcommand{\paragraph}{%
  \@startsection{paragraph}{4}{\z@}%
    {\baselineskip}{-0.5em}%
    {\normalfont\bfseries}}
\makeatother

\newtheorem{theorem}{Theorem}
\newtheorem*{theorem*}{Theorem}
\newtheorem{lemma}[theorem]{Lemma}
\newtheorem*{lemma*}{Lemma}
\newtheorem{prop}[theorem]{Proposition}

\newtheorem{corollary}[theorem]{Corollary}

\newtheorem*{corollary*}{Corollary}

\theoremstyle{definition}

\newtheorem{remark}[theorem]{Remark}
\newtheorem*{remark*}{Remark}
\newtheorem{example}[theorem]{Example}

\newtheorem{question}[theorem]{Question}
\newtheorem*{problem*}{Problem}

\newtheorem*{conj*}{Conjecture}

\newcommand{\z}{\raisebox{0.3ex}{$\cdot$}}   

\begin{document} 

\title{Clique number and triangle densities in $C_4$-free graphs}  

\author{Gunnar Fl\o ystad}
\address{Gunnar Fl\o ystad,
Matematisk institutt,
Universitetet i Bergen,
Bergen, Norway.}
\email{gunnar.floystad@uib.no}

\author{Andreas F. Holmsen}
\thanks{The second author was supported by the Institute for Basic Science (IBS-R029-C1)}
\address{Andreas F. Holmsen,
Department of Mathematical Sciences,
KAIST,
Daejeon, South Korea 
 \and 
Discrete Mathematics Group,  Institute for Basic Sciences (IBS), Daejeon, South Korea. }
\email{andreash@kaist.edu}

\date{\today}

\begin{abstract} 
For a $C_4$-free graph $G$ on $n$ vertices --- one with no induced cycle on
four vertices --- we study the two-sided extremal problem for the triangle
density $\tau$: How large and how small can $\tau$ be for given edge density
$\varepsilon$ and clique-number density $\kappa = \omega(G)/n$? 

We give lower and upper bounds for $\tau$ in terms of $\kappa$ and $\varepsilon$.
The two bounds sandwich $\tau$, and their compatibility forces a lower bound
for $\kappa$ in terms of $\varepsilon$. When the clique complex of $G$ is $2$-Leray
over a field $\Bbbk$, the resulting bound on the clique-number density lies between the previous best $C_4$-free bound and
the sharp chordal bound. It improves on the former {\it for every} $\varepsilon \in (0,1)$.

The lower bound is elementary. The upper bound is homological,
obtained by passing to the Stanley--Reisner ring of the clique complex. When the
complex is $2$-Leray, its Betti table has at most two linear strands.
The two first entries in the first strand encode edge and triangle densities, and the
strong structural form of a Boij--Söderberg decomposition constrains what
these entries can be, yielding the upper bound.

For $2$-Leray graphs with no holes in the range $[4,g]$ we give a conjecturally sharp bound. We further ask questions concerning the triangle bound for any $C_4$-free graph.
\end{abstract}

\maketitle 

\section{Introduction}

\subsection*{The clique-number density problem}

How large a clique must a dense graph contain if induced squares are
forbidden? For a general graph, positive edge density does not force a
clique containing a positive proportion of the vertices. Indeed, if
$G$ is a graph on $n$ vertices with \emph{edge density}
\[
    \varepsilon
    =
    \frac{|E(G)|}{\binom{n}{2}},
\]
then Tur\'an's theorem \cite{Turan1941} gives a lower bound on the
\emph{clique number} $\omega(G)$, the order of the largest complete
subgraph of $G$, depending on $\varepsilon$; but for fixed
$\varepsilon\in(0,1)$ this bound remains independent of $n$.
Consequently, Tur\'an's theorem gives no positive lower bound on the
\emph{clique-number density}
\[
    \kappa=\frac{\omega(G)}{n}
\]
as $n\to\infty$.

The situation changes when induced subgraphs are restricted. We say that a graph
is \emph{$C_4$-free} if it contains no induced cycle on four
vertices. Erd\H{o}s asked whether, for every $\varepsilon>0$, there
exists $f(\varepsilon)>0$ such that every $C_4$-free graph of edge
density $\varepsilon$ satisfies
\begin{equation}\label{eq:C4 free}
    \kappa\geq f(\varepsilon).
\end{equation}
Thus forbidding a single induced configuration on four
vertices would force a global clique of linear size.

Gy\'arf\'as, Hubenko, and Solymosi \cite{GyarfasHubenkoSolymosi2002} answered Erd\H{o}s's question
affirmatively, showing that \eqref{eq:C4 free} holds with
$f(\varepsilon) = \varepsilon^2/10$, and that the quadratic order is best possible as $\varepsilon\to0$. A later bound, due to
the second author \cite{Holmsen2020} and valid throughout the whole interval
$0<\varepsilon<1$, is
\begin{equation}\label{eq:holmsen}
    \kappa
    \geq
    \bigl(1-\sqrt{1-\varepsilon}\bigr)^2.
\end{equation}
The optimal function $f$ is nevertheless unknown.

The role of $C_4$ here is not incidental. Gy\'arf\'as, Hubenko, and Solymosi
also show
\cite[Proposition 1]{GyarfasHubenkoSolymosi2002} that for any graph $H$ which is
not an induced subgraph of $C_4$ there exist $H$-free graphs on $n$ vertices with
at least $n^2/4$ edges and clique number $o(n)$. Indeed, if $H$ has three pairwise
non-adjacent vertices, a Ramsey graph containing no three independent vertices
will serve; otherwise a balanced complete bipartite graph does. 
Thus among all single forbidden induced subgraphs $C_4$ is the unique maximal one for which there is a
bound of the form \eqref{eq:C4 free}.

A useful comparison is provided by chordal graphs. Recall that a graph
is \emph{chordal} if it contains no induced cycle of length greater
than three. For chordal graphs the sharp estimate (see e.g. \cite{AbbottKatchalski1979, GyarfasHubenkoSolymosi2002}) is
\begin{equation}\label{eq:chordal}
    \kappa
    \geq
    1-\sqrt{1-\varepsilon}.
\end{equation}
Since every
chordal graph is $C_4$-free, the extremal clique-density curve for
$C_4$-free graphs must lie between the chordal bound \eqref{eq:chordal}
and the best presently known general $C_4$-free bound
\eqref{eq:holmsen}. The aim of this paper is to narrow this gap for a
natural homologically defined subclass of $C_4$-free graphs, and to
introduce a method relating the extremal graph problem to the Betti
table of the clique complex.

The two classes are in fact endpoints of a spectrum. Chordality forbids induced cycles
of every length greater than three; forbidding only those of length at most $g$ 
gives an intermediate class (with the same homological assumption), and in Proposition~\ref{prop:holes} we prove 
the bound 
\[\kappa \geq 1-\sqrt{g/(g-2)}\,\sqrt{1-\varepsilon}.\] 

\subsection*{Triangle density as an intermediate parameter}

Our approach is to study the \emph{triangle density} rather than the
clique number directly. Write
\[
    \tau
    =
    \frac{(\text{number of triangles in }G)}{\binom{n}{3}}.
\]
We prove
inequalities of the form
\[
    \tau_-(\varepsilon,\kappa)
    \leq
    \tau
    \leq
    \tau_+(\varepsilon,\kappa).
\]
The lower estimate is combinatorial and holds for every $C_4$-free
graph. The upper estimate is homological and is proved when the clique
complex of $G$ is \emph{$2$-Leray over $\Bbbk$}, that is, when all its
induced subcomplexes have vanishing reduced homology with coefficients in
$\Bbbk$ in dimensions two and higher. 
(Note that $1$-Leray graphs are precisely the chordal graphs, by Fr\"oberg's theorem
\cite[Thm.9.2.3]{HerzogHibi}, \cite{Fro90}.)
Here and throughout, $\Bbbk$ is a fixed but
arbitrary field; homology, Betti numbers, the Leray condition, and the
Stanley--Reisner ring are all taken over it. Any field will do, provided the same
one is used throughout.
Since the same triangle density must satisfy both estimates,
the compatibility condition
$\tau_-(\varepsilon,\kappa)\leq \tau_+(\varepsilon,\kappa)$ forces a lower bound of the form \eqref{eq:C4 free}, that is, a lower bound on $\kappa$ in terms of $\varepsilon$.

\medskip

\noindent \emph{The lower bound.} The combinatorial mechanism behind the lower estimate is particularly
simple. If $u$ and $w$ are nonadjacent vertices of a $C_4$-free graph,
then their common neighbors form a clique: two nonadjacent common
neighbors would form an induced $C_4$ together with $u$ and $w$.
Thus the number of induced $2$-paths with a given nonadjacent pair of
endpoints is controlled by $\omega(G)$. Combining this observation
with elementary identities for induced subgraphs on three vertices
yields the following.

\begin{theorem}\label{thm:triangle lower bound}
Let $G$ be a graph on $n \geq 3$ vertices with edge density $\varepsilon$,
triangle density $\tau$, and clique-number density $\kappa$. If $G$ is
$C_4$-free, then
\begin{equation}\label{eq:lower}
    \tau
    \geq
    \max\left\{\,
        \varepsilon^{2}
        -\frac{(1-\varepsilon)(\kappa n+\varepsilon)}{n-2},
        \;\;
        1-\tfrac{3}{2}(1-\varepsilon)
          \left(1+\frac{\kappa n}{n-2}\right)
    \,\right\}.
\end{equation}
\end{theorem}

Here $\kappa n=\omega(G)$ is an integer, so \eqref{eq:lower} involves no rounding;
it is an exact inequality valid at every $n\geq 3$, and as $n\to\infty$ its
right-hand side increases to
\begin{equation}\label{eq:lower limit}
    \max\left\{\,
        \varepsilon^2-\kappa(1-\varepsilon),
        \;\;
        1-\tfrac{3}{2}(1+\kappa)(1-\varepsilon)
    \,\right\}.
\end{equation}
We have chosen to record the finite form \eqref{eq:lower} rather than
\eqref{eq:lower limit} with an additive $o(1)$, since the two are not equivalent
and only the former is established by the proof. The normalization factors
$\tfrac{n}{n-2}$ that occur along the way approach their limits from the wrong
side, so passing to the limit prematurely would leave the finite statement
unproved. The same convention is used for the upper bound.

\medskip

\noindent \emph{The upper bound.} The upper estimate comes from the clique complex $X(G)$. Its
Stanley--Reisner ideal is the edge ideal of the complement graph. 
By Hochster's formula, the absence of induced four-cycles is equivalent
to the vanishing of a single graded Betti number, $\beta_{2,4}$.
If $X(G)$
is $2$-Leray over $\Bbbk$, then all its induced subcomplexes have vanishing
reduced homology in dimensions at least two, and the Betti table of its
Stanley--Reisner ring has at most two linear strands. The clique
number also has an algebraic interpretation,
$\dim\Bbbk[X(G)]=\omega(G)$, so the clique-number density controls the
codimension of the Stanley--Reisner ring. We use a Boij--S\"oderberg
decomposition of the Betti table to convert the restrictions imposed
by $C_4$-freeness, the $2$-Leray condition, and the codimension into
an inequality between the first two entries of the first linear
strand. Translating this inequality back into graph-theoretic
quantities gives the following.

\begin{theorem}\label{thm:triangle upper bound}
Let $G$ be a graph on $n$ vertices with edge density $\varepsilon$,
triangle density $\tau$, and clique-number density $\kappa$. If $G$ is
$C_4$-free and its clique complex is $2$-Leray over $\Bbbk$, then
\begin{equation}\label{eq:upper}
    \tau
    \leq
    \frac{(3-\kappa^2)\,\kappa\,\varepsilon}
         {1+2\kappa-\kappa^2}.
\end{equation}
\end{theorem}

The two estimates are reached by independent routes:
\[
\begin{tikzcd}[column sep=large]
    \text{$C_4$-freeness and clique number}
    \arrow[r, "\text{counting}"]
    &
    \text{lower bound on $\tau$}
    \\
    \text{Betti table of $X(G)$}
    \arrow[r, "\text{Boij--S\"oderberg}"]
    &
    \text{upper bound on $\tau$.}
\end{tikzcd}
\]
Applied to the same graph they sandwich the triangle density, and their
incompatibility when $\kappa$ is too small forces a lower bound on $\kappa$ ---
the clique-density bound worked out in the next subsection. Whether the homological condition in
Theorem~\ref{thm:triangle upper bound} is a genuine restriction or an artifact of
the proof is discussed below.

\subsection*{Consequences}

Combining the two triangle-density estimates gives an implicit lower
bound on $\kappa$ as a function of $\varepsilon$ for $C_4$-free graphs
whose clique complexes are $2$-Leray over $\Bbbk$. Unlike the two
triangle-density bounds themselves, the resulting clique-density bound is
inherently asymptotic, since the normalization factors of \eqref{eq:lower}
survive the elimination of $\tau$. We therefore state it for the limiting form
\eqref{eq:lower limit}, so that the conclusions below hold up to an additive
$o(1)$ in $\kappa$. The two branches of \eqref{eq:lower limit} yield
distinct relations. For lower clique densities the dominating relation is
\begin{equation}\label{eq:sandwich low}
    \varepsilon^2 - \kappa(1-\varepsilon)
    \leq
    \frac{(3-\kappa^2)\,\kappa\,\varepsilon}{1+2\kappa - \kappa^2},
\end{equation}
while for higher clique densities we get
\begin{equation}\label{eq:sandwich high}
    1-\tfrac{3}{2}(1+\kappa)(1-\varepsilon)
    \leq
    \frac{(3-\kappa^2)\,\kappa\,\varepsilon}{1+2\kappa - \kappa^2}.
\end{equation}
Solving each relation for $\kappa$ gives a piecewise lower bound on the
clique-number density. The resulting curve lies strictly between the
general $C_4$-free bound \eqref{eq:holmsen} and the sharp chordal bound
\eqref{eq:chordal} for every $\varepsilon \in (0,1)$. 
In this sense the new estimate interpolates between
the known bounds for chordal graphs and for general $C_4$-free graphs.
Figure~\ref{fig:bounds} shows the resulting curve together with these
two bounds; the details are worked out in Section~\ref{sec:discussion}.

\begin{figure}[htbp]
    \centering
\begin{tikzpicture}[scale = 0.85]
\begin{axis}[
    width=0.8\linewidth, height=0.8\linewidth,
    xmin=0, xmax=1.05, ymin=0, ymax=1.05,
    axis lines=middle,
    xlabel={$\varepsilon$},
    ylabel={$\kappa$},
    xlabel style={at={(axis description cs:1.04, 0.0)},anchor=north east},
    ylabel style={at={(axis description cs:-0.02,0.99)},anchor=south,rotate=0},
    xtick={0,0.2,0.4,0.6,0.8,1.0},
    ytick={0,0.2,0.4,0.6,0.8,1.0},
    minor tick num=1,
    grid=both,
    major grid style={line width=.3pt,draw=gray!40},
    minor grid style={line width=.15pt,draw=gray!20},
    tick label style={font=\small},
    label style={font=\small},
    samples=200,
    clip=true,
]
\addplot[cChordal, thick, domain=0:1] {1 - sqrt(1-x)};
\addplot[cHolmsen, thick, domain=0:1] {(1 - sqrt(1-x))^2};
\addplot[cLow, very thick, samples=200, domain=0.0001:0.221992, variable=\k]
    ( { (sqrt(\k)*sqrt(\k^4 - 3*\k^3 + 5*\k + 1) - \k*(\k-1)) / (1 + 2*\k - \k*\k) } , \k );
\addplot[cLow, thin, dash pattern=on 0.5pt off 1.2pt, samples=160, domain=0.221992:0.62, variable=\k]
    ( { (sqrt(\k)*sqrt(\k^4 - 3*\k^3 + 5*\k + 1) - \k*(\k-1)) / (1 + 2*\k - \k*\k) } , \k );
\addplot[cHigh, very thick, samples=200, domain=0.221992:0.9999, variable=\k]
    ( { (-3*\k^3 + 5*\k^2 + 5*\k + 1) / (-\k^3 + 3*\k^2 + 3*\k + 3) } , \k );
\addplot[cHigh, thin, dash pattern=on 0.5pt off 1.2pt, samples=160, domain=0.0:0.221992, variable=\k]
    ( { (-3*\k^3 + 5*\k^2 + 5*\k + 1) / (-\k^3 + 3*\k^2 + 3*\k + 3) } , \k );
\end{axis}
\end{tikzpicture}
    \caption{Lower bounds on the clique-number density $\kappa$ in terms of edge density $\varepsilon$. The new bound for induced-$C_4$-free graphs with $2$-Leray clique complex improves the previous general induced-$C_4$-free bound (in red) and lies below the sharp chordal bound (in blue). The two pieces arise from two different lower estimates for the triangle density.}
    \label{fig:bounds}
\end{figure}
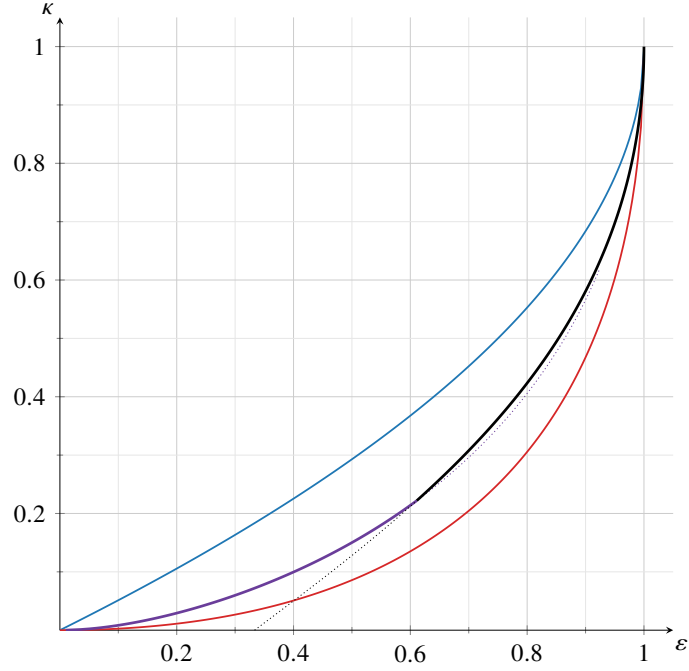

For small edge density, the active branch of our estimate has the
expansion
\[
    \kappa
    =
    \varepsilon^2
    -2\varepsilon^3
    +4\varepsilon^4
    +O(\varepsilon^5).
\]
In particular $\kappa\sim\varepsilon^2$ as $\varepsilon\to0$, with
leading coefficient $1$. The quadratic order agrees with the known
extremal behavior for general $C_4$-free graphs, while the coefficient
agrees with the refined small-density estimate of Gy\'arf\'as,
Hubenko, and Solymosi \cite[Theorem 2]{GyarfasHubenkoSolymosi2002}. 
At the other end of the interval, our bound is
asymptotic to $\kappa = 1-\sqrt{2}\sqrt{1-\varepsilon}$ as
$\varepsilon\to1$.

\begin{example}\label{ex:circulant graph}
The representative case $\varepsilon = \tfrac12$ is realized by a family of
circulants. Let $G_k = C^k_{4k+1}$ denote the $k$th power of the cycle on $4k+1$ vertices,
that is, the circulant graph in which two vertices are adjacent whenever their
indices differ by at most $k$ modulo $4k+1$. Each $G_k$ is $C_4$-free and
$2k$-regular, so its edge density satisfies $\varepsilon \to \tfrac12$ as
$k \to \infty$. Its clique number is $\omega(G_k) = k+1$, giving clique-number
density $\kappa \to \tfrac14$. These graphs, introduced in
\cite{GyarfasHubenkoSolymosi2002}, are the best construction known for the
clique-number problem at edge density $\tfrac12$: they show that a $C_4$-free
graph of edge density $\tfrac12$ need not have clique-number density exceeding
$\tfrac14$. Gy{\'a}rf{\'a}s and S{\'a}rk{\"o}zy \cite{GyarfasSarkozy2017} proved
that these graphs are extremal within the class of \emph{regular} graphs, confirming a
conjecture of \cite{GyarfasHubenkoSolymosi2002}: every $2k$-regular $C_4$-free
graph on $4k+1$ vertices contains a clique of size $k+1$. Whether $\kappa =
\tfrac14$ is optimal at $\varepsilon = \tfrac12$ over all $C_4$-free graphs
remains open, and we return to this point in
Section~\ref{sec:discussion}. The case $k=3$ is drawn in
Figure~\ref{fig:G3}.

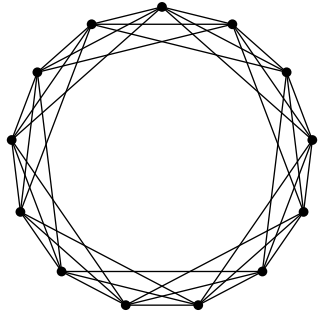
\begin{figure}[htbp]
  \centering
\begin{tikzpicture}[
    vtx/.style={circle,fill=black,inner sep=1.3pt}
  ]
  \def\n{13}
  \def\R{2}
  \foreach \i in {0,...,12}{
    \pgfmathsetmacro{\ang}{90 - 360*\i/\n}
    \coordinate (v\i) at (\ang:\R);
  }
  \foreach \d in {1,2,3}{
    \foreach \i in {0,...,12}{
      \pgfmathtruncatemacro{\j}{mod(\i+\d,\n)}
      \draw[black,line width=0.5pt] (v\i) -- (v\j);
    }
  }
  \foreach \i in {0,...,12}{ \node[vtx] at (v\i) {}; }
\end{tikzpicture}
\caption{The circulant graph $G_3 = C_{13}^{\,3}$.}
  \label{fig:G3}
\end{figure}
\end{example}

At this edge density our estimate gives $\kappa\geq 0.1501\ldots$ under the
$2$-Leray hypothesis, improving the previous general lower bound
\eqref{eq:holmsen}, which reads $(1-1/\sqrt{2})^2 = 0.0857\ldots$ here. A substantial gap remains between
our lower bound and the value $1/4$ attained by the circulant
construction, so determining the true extremal value at
$\varepsilon=1/2$ is already an interesting special case of the general
problem.

\subsection*{The Leray hypothesis and open questions}

The $2$-Leray condition is essential to our proof of the upper
triangle-density bound: it ensures that the relevant Betti table has at
most two linear strands. It is not clear, however, whether this
condition reflects a genuine extremal distinction or merely a
limitation of the method. We are not aware of any $C_4$-free graph
violating \eqref{eq:upper}, even when its clique complex has
nonvanishing homology in dimension two or higher. This leads to the
principal structural question raised by the paper:
\begin{question} \label{qu:tau bound}
    Does the upper triangle-density bound \eqref{eq:upper} hold for
    every $C_4$-free graph, without assuming that its clique complex is
    $2$-Leray?
\end{question}
An affirmative answer would make the resulting clique-density bound
valid for all $C_4$-free graphs. More generally, the method reduces the
clique-density problem to a question about which Betti tables can occur
for clique complexes of $C_4$-free graphs. Thus the extremal problem is
connected not only to subgraph counting, but also to the geometry of
the cone of Betti tables and to the homology of flag complexes.

\subsection*{Organization of the paper}

Section~\ref{sec:prelim} introduces the densities and counting
identities used throughout, and recalls the connection between the
clique complex and its Stanley--Reisner ring.
Section~\ref{sec:resolutions} reviews the part of Boij--S\"oderberg
theory needed for the argument. In Section~\ref{sec:upper} we prove the
upper bound on the triangle density under the $2$-Leray hypothesis.
Section~\ref{sec:lower} gives the elementary lower bound for all
$C_4$-free graphs, by combinatorial arguments similar to those in
\cite{GyarfasHubenkoSolymosi2002, Holmsen2020}. Finally, Section~\ref{sec:discussion} combines the
estimates, compares the resulting clique-density bound with known
bounds, discusses the case $\varepsilon=1/2$, recovers the sharp
chordal estimate, and formulates the remaining open problems.

\section{Preliminaries, notation, and terminology}\label{sec:prelim}

Throughout we fix a field $\Bbbk$; all homology groups and Betti numbers are taken
over $\Bbbk$. We continue to write $\varepsilon$, $\kappa$,
and $\tau$ for the edge, clique-number, and triangle densities of a
graph $G$ on $n$ vertices. We also write $m = 1-\varepsilon$ for the missing-edge density.
For a subset $S\subseteq V$ we write $G[S]$ to denote the subgraph of $G$ induced on $S$. The same notation is also applied to the induced subcomplexes of a simplicial complex.

\paragraph{Induced subgraphs of order three}

Let $G = (V,E)$ be a graph on $n$ vertices. We partition the three-element
subsets of $V$ according to how many edges they span,
\[
  \binom{V}{3} = T_0 \cup T_1 \cup T_2 \cup T_3,
  \qquad
  T_i = \left\{ S \in \tbinom{V}{3} \;:\; G[S] \text{ contains exactly } i \text{ edges} \right\}.
\]
Thus $T_3$ is the set of triangles of $G$, and $T_2$ is the set of induced
$2$-paths (paths on three vertices). The following counting identities are a
basic tool of this paper.

\begin{lemma}\label{lem:census}
For any graph $G=(V,E)$ on $n$ vertices,
\begin{enumerate}\setlength{\itemsep}{.4em}
  \item $|T_0| + |T_1| + |T_2| + |T_3| = \dbinom{n}{3}$,
  \item $|T_1| + 2|T_2| + 3|T_3| = (n-2)\,|E|$,
  \item $|T_2| + 3|T_3| = \displaystyle\sum_{v\in V} \binom{\deg v}{2}$.
\end{enumerate}
\end{lemma}

\begin{proof}
Identity (1) holds by definition. Identity (2) is a double count of the incident
pairs $(e,S)$ with $e$ an edge, $S$ a triple containing $e$: summing over triples,
a triple in $T_i$ contributes $i$; summing over edges, each edge lies in exactly
$n-2$ triples. For identity (3), consider the (open) neighborhood $N_v$ of a vertex $v$.
Each edge of $G[N_v]$ gives a triangle of $G$ through $v$, and each non-edge of
$G[N_v]$ gives an induced $2$-path centered at $v$. Summing over $v$ counts each
triangle three times (once per vertex) and each induced $2$-path once (at its
center).
\end{proof}

\paragraph{The clique complex and its Stanley--Reisner ring}

Let $X = X(G)$ denote the \emph{clique complex} of $G$, the simplicial complex whose
faces are the cliques of $G$. Its Stanley--Reisner ideal is generated by the
minimal non-faces of $X$, which are exactly the non-edges of $G$; hence this
ideal is the edge ideal $I(\overline{G})$ of the complement graph $\overline{G}$, and the
Stanley--Reisner ring of $X$ is $\Bbbk[X] = S/I(\overline{G})$, where
$S = \Bbbk[x_v : v \in V]$.

We record the graded Betti numbers of this ring that we will use.
Recall Hochster's formula \cite[Corollary 5.12]{MillerSturmfels}, 
which expresses the graded Betti numbers of a
Stanley--Reisner ring in terms of the reduced homology of the induced
subcomplexes of $X$: for a simplicial complex $X$ on vertex set $V$,
\begin{equation}\label{eq:hochster}
  \beta_{i,j}\bigl(\Bbbk[X]\bigr)
  \;=\;
  \sum_{\substack{W \subseteq V \\ |W| = j}}
  \dim_{\Bbbk} \tilde H_{\,j-i-1}\bigl(X[W]; \Bbbk\bigr),
\end{equation}
where $X[W]$ denotes the subcomplex of $X$ induced on $W$. 

\begin{remark} \label{rem:resolutions} We briefly recall the origin and significance of graded Betti numbers.
For the polynomial ring $S$ let $S(-j) = Su$ be the free $S$-module with generator
$u$ in degree $j$. For a graded ideal $I$ in $S$, the quotient ring $S/I$ (like
$\Bbbk[X]$) has 
a {\it minimal free resolution} \cite[Sec. 1.4]{MillerSturmfels}, unique up to isomorphism:
\[ S \leftarrow F_1 \leftarrow F_2 \leftarrow \cdots \leftarrow F_r \leftarrow 0.\]
where $F_i = \oplus_{j \in {\mathbb Z}} S(-j)^{\beta_{i,j}(S/I)}$ is a free graded
$S$-module (by $S(-j)^m$ we mean a direct sum of $m$ copies).

For later, the resolution is called {\it pure}, if for each $0 \leq i \leq r$ there
is exactly one $j$ with nonzero $\beta_{i,j}$. We then denote this $j$ as $d_i$, and thus get associated a {\it degree sequence}
$(d_0, d_1, \ldots, d_r)$ (with $d_0 = 0$).
\end{remark}

\begin{example}\label{ex:betti table G3}
As an illustration we consider the circulant graph $G_3 = C_{13}^{\,3}$ of
Example~\ref{ex:circulant graph}. The array below is the Betti table of
$\Bbbk[X(G_3)] = S/I(\overline{G_3})$, whose entries are given by \eqref{eq:hochster}; we computed it using Macaulay2 \cite{M2}. The rows are 
indexed by $d = j-i$ and the columns by the homological degree $i$. 
\[
\begin{array}{r|*{12}{c}}
   d & 0 & 1 & 2 & 3 & 4 & 5 & 6 & 7 & 8 & 9 & 10 & 11 \\ \hline
   0 & 1 & \z & \z & \z  & \z  & \z  & \z  & \z  & \z  & \z & \z & \z \\
   1 & \z & 39 & 182 & 377 & 390 & 195 & 39 & \z & \z & \z & \z & \z \\
   2 & \z & \z & \z & 39 & 273 & 663 & 819 & 598 & 273 & 78 & 13 & 1 \\
\end{array}
\]
The table has nonzero entries only in the rows $d = j-i \in \{0,1,2\}$; that is,
it has \emph{two} linear strands. By Hochster's formula \eqref{eq:hochster} this means
$\tilde H_p(X(G_3)[W];\Bbbk) = 0$ for every induced subcomplex $X(G_3)[W]$ and all
$p \geq 2$, so $X(G_3)$ is $2$-Leray.

The same holds for every circulant graph $G_k$ of
Example~\ref{ex:circulant graph}. In fact, the clique complex $X(G_k)$ can be realized as the nerve complex of a family of evenly distributed arcs on the unit circle, each of length $\frac{2\pi k}{4k+1}$. It is then a simple consequence of the nerve theorem that $X(G_k)$ is $2$-Leray. 
\end{example}

Writing
$\beta_{i,j}$ for the entries of the Betti diagram of
$\Bbbk[X(G)] = S/I(\overline{G})$, the two \emph{density parameters} we read off it
are
\[
  m = \frac{\beta_{1,2}}{\binom{n}{2}} = 1 - \varepsilon,
  \qquad
  b = \frac{\beta_{2,3}}{\binom{n}{3}} =  \frac{2|T_0| + |T_1|}{\binom{n}{3}},
\]
obtained by scaling each entry by the binomial coefficient of its internal degree.
The first equals $m$ because $\beta_{1,2}$ counts the generators of the edge
ideal (the non-edges of $G$). Both $\beta_{1,2}$ and $\beta_{2,3}$ lie in the first
linear strand $d = j-i = 1$ of the diagram, $\beta_{1,2}$ being its first entry and
$\beta_{2,3}$ its second, as one sees in the table of $G_3$ in
Example~\ref{ex:betti table G3}, where $\beta_{1,2} = 39$ and $\beta_{2,3} = 182$.
(These density parameters are not to be confused with the Boij--S\"oderberg
normalization to be used in Section~\ref{sec:resolutions}, which instead scales a whole
diagram by a single constant so that its top-left entry equals $1$.)

Normalizing the triple counts by $|T_i| = t_i \binom{n}{3}$, Lemma~\ref{lem:census}
becomes
\begin{equation}\label{eq:normalized census}
  t_0 + t_1 + t_2 + t_3 = 1,
  \qquad
  t_1 + 2t_2 + 3t_3 = 3(1-m),
  \qquad
  t_2 + 3t_3 = D,
\end{equation}
where $D = \binom{n}{3}^{-1}\sum_{v} \binom{\deg v}{2}$. 
Solving the first two equations of \eqref{eq:normalized census} for $t_2$ and
$t_3$, and using the third together with $b = 2t_0 + t_1$, gives the identities
\begin{equation}\label{eq:t2 and t3}
  t_2 = D - 3t_3 = 3m - 3t_0 - 2t_1,
  \qquad
  \tau = t_3 = b - 3m + 1,
\end{equation}
which we use to pass between the triangle density $\tau$, the triple count, and the
density parameter $b$.

\section{Boij--S\"oderberg decompositions}\label{sec:resolutions}

Boij--S\"oderberg theory asserts that the Betti diagram of a graded module over
$S$ is a positive rational combination of \emph{pure} diagrams, indexed by a chain
of strictly increasing integer sequences \cite{eng, flo}. In this section we
recall the pure diagrams we need and their normalized Betti numbers; the material
is standard and we include it to fix notation.

\paragraph{Pure diagrams and their types}

Let $r$ be a positive integer and let $\mathbf{d} = (d_0, d_1, \dots, d_r)$ be a
strictly increasing sequence of integers, which we call a \emph{type} of \emph{length} $r$. 
(This is just another name for the degree sequence introduced in Remark \ref{rem:resolutions}.)
Types are partially ordered by $\mathbf{d} \preceq \mathbf{e}$ if $\mathbf{d}$
is at least as long as $\mathbf{e}$, and $d_i \leq e_i$ for each $i$ up to the length 
of $\mathbf{e}$.

A type $\mathbf{d}$ is the sequence
of generating degrees of a {\it pure} resolution whose $p$th free module is generated
in degree $d_p$. The corresponding {\it pure} Betti diagram of $\mathbf{d}$ 
is determined up to scalar multiple, and has a single nonzero entry in
each column $p$, namely in internal degree $d_p$; we index it
$\hat\beta_{p,d_p}$, placing it in column $p$ and row $d_p - p$. The
\emph{normalized pure diagram of type $\mathbf{d}$}, written $\pi(\mathbf{d})$,
is the diagram with entries, \cite[Sec.1.4]{flo}:
\begin{equation}\label{eq:normbetti}
  \hat\beta_{p,d_p}
  = \frac{\prod_{0 < k \leq r}(d_k - d_0)}
         {\prod_{0 \leq k < p}(d_p - d_k)\cdot \prod_{p < k \leq r}(d_k - d_p)}.
\end{equation}
It is normalized so that $\hat\beta_{0,d_0} = 1$. Observe that the entries are invariant under
translation of $\mathbf{d}$; we may therefore assume $d_0 = 0$ throughout. We
write $\hat\beta_{i,j}(\mathbf{d})$ when the type is not clear from context.
The shape of the diagram --- which entry sits in
which row --- is read off from $\mathbf{d}$ as well, since the entry in column
$p$ lies in row $d_p - p$.

The significance of pure Betti diagrams is that any Cohen-Macaulay module $M$
of codimension $r$ with pure resolution of type $\mathbf{d}$ has Betti diagram 
$\lambda \pi({\mathbf{d}})$ for some scalar $\lambda$.

\begin{remark} This fact has a standard simple argument. 
The Herzog--K\"uhl equations, \cite[Sec.1.3]{flo}, form a linear system of
rank $r$ in the $r+1$ unknowns $\hat\beta_{p,d_p}$, so the solution space is
one-dimensional, and \eqref{eq:normbetti} selects the solution with
$\hat\beta_{0,d_0} = 1$.
\end{remark}

\begin{remark}\label{rem:normalization}
Normalizations differ in the literature. Our $\pi(\mathbf{d})$ agrees with the
pure table of \cite{eng}, scaled so that the top-left entry is $1$, whereas
\cite{flo} takes $\pi(\mathbf{d})$ to be the smallest \emph{integer} solution of
the Herzog--K\"uhl equations. The two differ by a rational factor. They coincide
for $\mathbf{d} = (0,2,3)$, where both give $(1,3,2)$, but for
$\mathbf{d} = (0,2,3,4,5,7,8,\dots,14)$ ours begins $(1, \tfrac{182}{3}, 364, \dots )$
while that of \cite{flo} begins $(3, 182, 1092, \dots )$. Since we work throughout with
convex combinations of diagrams whose top-left entry is $1$, the present
normalization is the convenient one.
\end{remark}

\subsection*{Betti diagrams of edge ideals}

We apply Boij--S\"oderberg theory to the Stanley--Reisner ring
$\Bbbk[X(G)] = S/I(\overline{G})$ of the clique complex, and write $B(G)$ for its
normalized Betti diagram. The theory says that $B(G)$ decomposes into an
\emph{convex combination} of normalized pure diagrams, \cite{BS2012} or \cite[Sec.5]{flo}:
\begin{equation}\label{eq:BS}
  B(G) = \sum_{i} \alpha_i\, \pi\bigl(\mathbf{d}^{(i)}\bigr),
  \qquad \alpha_i > 0, \quad \sum_i \alpha_i = 1,
\end{equation}
whose types form a chain
$\mathbf{d}^{(1)} \prec \mathbf{d}^{(2)} \prec \cdots$. Here \emph{normalized} is
meant in the sense of this section: a diagram is scaled by a single rational
constant so that its top-left entry $\beta_{0,0}$ equals $1$. For an edge ideal
this scaling is trivial, as $\beta_{0,0} = 1$ already, so $B(G)$ is the ordinary
integral Betti diagram; the pure summands $\pi(\mathbf{d}^{(i)})$, however, may
have rational entries.

The \emph{length} of a pure diagram of type $\mathbf{d} = (d_0,\dots,d_r)$ is
$r$, and the lengths in \eqref{eq:BS} are bounded below in terms of the clique
number. Indeed, the Krull dimension of a Stanley--Reisner ring is one more than
the dimension of the complex, so
\[
  \dim \Bbbk[X(G)] = \dim X(G) + 1 = \bigl(\omega(G) - 1\bigr) + 1 = \omega(G),
\]
since the top-dimensional faces of $X(G)$ are the maximum cliques of $G$. Every
pure diagram in a Boij--S\"oderberg decomposition of a quotient ring $\Bbbk[X(G)]$ 
has length at least the
codimension of the quotient ring, and the shortest one attains it, so, writing $r_i$ for the length of
$\mathbf{d}^{(i)}$,
\begin{equation}\label{eq:min length}
  \min_i r_i = n - \dim\Bbbk[X(G)] = n - \omega(G) = (1-\kappa)\,n.
\end{equation}
In particular $r_i \geq (1-\kappa)n$ for every $i$.

The edge-ideal structure further restricts the types appearing in
\eqref{eq:BS}. Writing $\mathbf{d} = (0, d_1, \dots, d_r)$:
\begin{itemize}
  \item $d_1 \geq 2$ (the first row of $B(G)$ is $(1,0,\dots,0)$);
  \item $d_k \leq 2k$ for $1 \leq k \leq r$, so $B(G)$ vanishes below the main
        diagonal. This is true for the edge ideal of \emph{any} graph; taking $k = 1$ this forces $d_1 = 2$;
  \item if $\overline{G}$ has no $2K_2$, equivalently $G$ is $C_4$-free, then
        $d_2 = 3$ (assuming $r \geq 2$);
  \item more generally, call an induced cycle of length greater than three a
        \emph{hole}. If $G$ has no hole of length less than $h$, then every type
        in \eqref{eq:BS} satisfies $d_p - p \leq 1$ for every $p$ with
        $d_p < h$. (The previous item is the case $h = 5$.)
\end{itemize}
The first three items are immediate. The last is a theorem of Eisenbud, Green,
Hulek, and Popescu \cite[Theorem 2.1]{EGHP2005}: the Stanley--Reisner ideal of
the clique complex of $G$ has a linear resolution for $p$ steps exactly when
every hole of $G$ has length at least $p+3$, so a nonlinear entry first appears
in internal degree equal to the length of the shortest hole. Their result is
independent of $\Bbbk$ \cite[Corollary 2.7]{EGHP2005}, and specializes to
Fr\"oberg's characterization of the chordal case \cite[Corollary 2.2]{EGHP2005}.
In particular the bound is sharp: if $W$ is the vertex set of a shortest hole
then $X(G)[W]$ is a circle, so row $2$ of $B(G)$ begins in internal degree
exactly the length of that hole.

\paragraph{The two-strand types $\pi_{a,L}$}

We now narrow to the pure diagrams that enter the upper bound in
Section~\ref{sec:upper}: those with at most two linear strands, that is, with
nonzero entries only in rows $d = j - i \in \{0,1,2\}$. Since the entry in column
$p$ lies in row $d_p - p$, this says exactly that
\begin{equation}\label{eq:two strands}
  d_p - p \leq 2 \qquad \text{for all } p.
\end{equation}
Together with the restrictions above, this determines the type completely, as we now record.

Consider then a type $\mathbf{d} = (0, d_1, \dots, d_L)$ with $L \geq 2$,
$d_1 = 2$ and $d_2 = 3$, satisfying \eqref{eq:two strands}. The entry in column
$p$ lies in row $d_p - p$, and $d_p - d_{p-1} \geq 1$, so this row is
non-decreasing in $p$; it already equals $1$ at $p = 1$ (as $d_1 = 2$) and never
exceeds $2$. Hence the row jumps from $1$ to $2$ at most once, at some index
$p_0 \geq 3$ (necessarily $p_0 \geq 3$ since $d_2 = d_1 + 1$), and the jump is by
exactly one row. So $d_p = p+1$ before the jump and $d_p = p+2$ from $p_0$ on,
which is to say
\begin{equation}\label{eq:dtype}
  \mathbf{d} \;=\; \mathbf{d}(a,L)
  \;:=\; (0,\,2,\,3,\,\dots,\,a-1,\,a+1,\,\dots,\,L+2),
\end{equation}
the sequence obtained from $0,2,3,\dots,L+2$ by deleting a single value $a =
p_0+1$, with $4 \leq a \leq L+2$; the value $a$ is recovered as the unique element
of $\{2,3,\dots,L+2\}$ missing from $\mathbf{d}$, so the presentation is unique.
The extreme value $a = L+2$ gives $(0,2,3,\dots,L+1)$, the type of a diagram with
a single linear strand.

For $a$ and $L$ as above we write $\pi_{a,L} := \pi(\mathbf{d}(a,L))$.
Its
two first-strand entries, computed from \eqref{eq:normbetti}, are
\begin{equation}\label{eq:strand entries}
  \hat\beta_{1,2}\bigl(\mathbf{d}(a,L)\bigr) = \frac{a-2}{a}\binom{L+2}{2},
  \qquad
  \hat\beta_{2,3}\bigl(\mathbf{d}(a,L)\bigr) = \frac{2(a-3)}{a}\binom{L+2}{3}.
\end{equation}
Both formulas remain valid at the one-strand end $a = L+2$, where they reduce to
$\binom{L+1}{2}$ and $\tfrac13 L(L^2-1)$. The one-strand diagrams are thus not a
separate case but the boundary of the same family, a fact we use again in
Section~\ref{sec:discussion}.

From here on we rescale lengths by $n$: for a type of length $L$ we write
$\ell = L/n \in (0,1]$ for its \emph{length density}, matching the scale of
$1-\kappa$ in \eqref{eq:min length}. To $\pi_{a,L}$ we attach two \emph{density
parameters}, obtained from the entries \eqref{eq:strand entries} by scaling by
$\binom n2$ and $\binom n3$ respectively, in the same way that $m$ and $b$ were
attached to $B(G)$ in Section~\ref{sec:prelim}:
\begin{equation}\label{eq:pure entries exact}
  m_{a,\ell}^{(n)} = \frac{a-2}{a}\cdot\frac{\binom{L+2}{2}}{\binom{n}{2}},
  \qquad
  b_{a,\ell}^{(n)} = \frac{2(a-3)}{a}\cdot\frac{\binom{L+2}{3}}{\binom{n}{3}},
  \qquad L = \ell n.
\end{equation}
As $n \to \infty$ with $\ell$ fixed these converge, and we write the limits as
\begin{equation}\label{eq:pure entries}
  m_{a,\ell} = \frac{a-2}{a}\,\ell^2,
  \qquad
  b_{a,\ell} = \frac{2(a-3)}{a}\,\ell^3.
\end{equation}
The limiting values are recorded for orientation only; the arguments of
Section~\ref{sec:upper} are carried out with the exact entries
\eqref{eq:pure entries exact}.

\begin{lemma} \label{lem:BS length} Let $G$ be a $C_4$-free graph with clique
complex $X(G)$ that is $2$-Leray over $\Bbbk$. In a Boij-S\"oderberg decomposition \eqref{eq:BS} the diagrams involved are $\pi_{a,L}$ where $4 \leq a \leq L+2$ and either:
\begin{itemize}
\item $n - \dim \Bbbk[X(G)] \leq L \leq n-2$ or,
\item $a = L+2$ and $L = n-1$.
In this case $m_{a,\ell}^{(n)} = 1$ and $b_{a,\ell}^{(n)} = 2$.
\end{itemize}
\end{lemma}

\begin{proof}
We first observe that $B(G)$ has at most two linear strands. Indeed, by
Hochster's formula \eqref{eq:hochster}, $\beta_{i,j}(\Bbbk[X(G)])$ is a sum of
dimensions of reduced homology groups $\tilde H_{j-i-1}(X(G)[W]; \Bbbk)$ over all 
induced subcomplexes on $j$ vertices. Since $X(G)$ is $2$-Leray, these vanish whenever
$j - i - 1 \geq 2$, so $\beta_{i,j} = 0$ unless $j - i \leq 2$. Thus the only
nonzero rows of $B(G)$ are $d = j - i \in \{0,1,2\}$; equivalently,
$\operatorname{reg}\bigl(\Bbbk[X(G)]\bigr) \leq 2$.

By Boij--S\"oderberg theory, $B(G)$ is a convex combination of pure diagrams
$\pi(\mathbf{d}^{(i)})$ as in \eqref{eq:BS}, each of which has at most two linear
strands, so each type satisfies \eqref{eq:two strands}. Since $G$ is $C_4$-free,
the restrictions of Section~\ref{sec:resolutions} give $d^{(i)}_1 = 2$ and
$d^{(i)}_2 = 3$ for every $i$. 
The classification of two-strand types after \eqref{eq:two strands} therefore applies
to each summand: writing $L_i$ for the length of $\mathbf{d}^{(i)}$, we have
$\pi(\mathbf{d}^{(i)}) = \pi_{a_i, L_i}$ for a unique $a_i$ with $4 \leq a_i \leq
L_i + 2$ (the value may differ from summand to summand, which is why it carries
the subscript $i$). 

By formula \eqref{eq:hochster}, $j \leq |V| = n$ for a nonzero $\beta_{i,j}(\Bbbk[X(G)])$. The same will then be true for all non-zero betti numbers occurring in the diagrams 
$\pi_{a_i,L_i}$ in the positive linear combination \eqref{eq:BS}.
Then by \eqref{eq:dtype} we must have $L_i+2 \leq n$ if $a_i \leq L_i+1$, or if $a_i = L_i+2$ we have $L_i+1 \leq n$.
The last statement follows by insertions into \eqref{eq:pure entries exact}.
\end{proof}

\begin{example}\label{ex:BS G3}
For the running example $G_3 = C_{13}^{\,3}$, whose Betti diagram $B(G_3)$ appears
in Example~\ref{ex:betti table G3}, the Boij--S\"oderberg decomposition
\eqref{eq:BS} has seven pure summands, each of the form $\pi_{a,L}$ of the classification above:
\[
\begin{aligned}
  B(G_3) \;=\;
    & \tfrac{1}{33}\,\pi_{8,11}
    + \tfrac{119}{1584}\,\pi_{7,11}
    + \tfrac{91}{1584}\,\pi_{7,10}
    + \tfrac{65}{264}\,\pi_{6,10} \\[2pt]
    & + \tfrac{39}{308}\,\pi_{6,9}
    + \tfrac{117}{308}\,\pi_{5,9}
    + \tfrac{13}{154}\,\pi_{4,9},
\end{aligned}
\]
a convex combination of normalized pure diagrams. The decomposition was computed in Macaulay2 \cite{M2} using the \texttt{BoijSoederberg} package. (Note that Macaulay2 lists the pure diagrams with the smallest integer solution to the Herzog--K{\"u}hl equations and adjusts the coefficients accordingly.) Every summand has $a \geq 4$,
as it must for a $C_4$-free graph, and the shortest summands have length
$9 = 13 - \omega(G_3)$, in agreement with \eqref{eq:min length} since
$\omega(G_3) = 4$.
\end{example}

\section{The upper bound}\label{sec:upper}

In this section we prove Theorem~\ref{thm:triangle upper bound}. We use the density
parameters $m$ and $b$ of Section~\ref{sec:prelim}, recalling that
$\tau = b - 3m + 1$ by \eqref{eq:t2 and t3}.
Set $\ell = 1 - \kappa$.

The key estimate is that the pure diagrams $\pi_{a,L'}$ with $a \geq 4$ and
$\ell' = L'/n \geq \ell$ all lie on or below a single line in the $(m,b)$-plane:
the line through:
\begin{itemize}
\item $(1,2)$, the value for the extremal length $L^\prime = n-1$ by Lemma \ref{lem:BS length}.
It is also the value for $L^\prime = n-2$, being the limit of
$\bigl(m^{(n)}_{a,\ell'}, b^{(n)}_{a,\ell'}\bigr)$ as $a \to \infty$,
\item $\bigl(\tfrac12\ell^2, \tfrac12\ell^3\bigr) = (m_{a,\ell},b_{a,\ell})$, the
limiting point of $\bigl(m^{(n)}_{4,\ell}, b^{(n)}_{4,\ell}\bigr)$ as $n \to \infty$.
\end{itemize}
Being pinned at these two extremes, the estimate
has no slack at either, so we prove it at the level of the \emph{exact} entries
\eqref{eq:pure entries exact}: an argument through the limiting values
\eqref{eq:pure entries} alone would not fix the direction of the inequality
there, and no passage to the limit is needed anywhere below.

\begin{lemma}\label{lem:halfplane} Let $0 \leq \ell \leq 1$.
Let $a, L, L', n$ be integers with $2 \leq L \leq L' \leq n-2$ and
$4 \leq a \leq L'+2$, and write $\ell = L/n$ and $\ell' = L'/n$, so that $\ell \leq \ell^\prime$. Then the exact entries \eqref{eq:pure entries exact} of
$\pi_{a,L'}$ satisfy
\begin{equation}\label{eq:halfplane}
  (2-\ell^2)\bigl(b_{a,\ell'}^{(n)} - 2\bigr)
  \;\leq\; (4-\ell^3)\bigl(m_{a,\ell'}^{(n)} - 1\bigr).
\end{equation}
\end{lemma}

\begin{proof}
Write $q = L+2$ and $q' = L'+2$ for the top degrees of the two types, and
abbreviate
\[
  Z(q') \;=\; \bigl(m_{a,\ell'}^{(n)},\, b_{a,\ell'}^{(n)}\bigr)
        \;=\; \Bigl(c_2\tbinom{q'}{2},\; c_3\tbinom{q'}{3}\Bigr),
  \qquad
  c_2 = \frac{a-2}{a\binom n2},
  \quad
  c_3 = \frac{2(a-3)}{a\binom n3},
\]
both constants being positive because $a \geq 4$. The region defined by
\eqref{eq:halfplane} is a half-plane, and it is closed downwards in $b$.

\emph{Step 1: reduction to the endpoints $q' = n$ and $q' = q$.}
The points $Z(4), Z(5), Z(6), \dots$ form a convex chain. Indeed
\[
  \binom{q'+1}{2}-\binom{q'}{2} = q',
  \qquad
  \binom{q'+1}{3}-\binom{q'}{3} = \binom{q'}{2},
\]
so the segment from $Z(q')$ to $Z(q'+1)$ has slope
\[
  \frac{c_3\binom{q'}{2}}{c_2\,q'} \;=\; \frac{c_3}{c_2}\cdot\frac{q'-1}{2},
\]
which increases strictly with $q'$. Since the first coordinate
$c_2\binom{q'}{2}$ also increases with $q'$, every $Z(q')$ with
$q \leq q' \leq n$ lies on or below the chord joining $Z(q)$ and $Z(n)$. If both
endpoints satisfy \eqref{eq:halfplane}, then so does $Z(q')$, since the half-plane is closed
downwards. It therefore suffices to treat $q' = n$ and $q' = q$.

\emph{Step 2: the endpoint $q' = n$.} Here
$\binom{q'}{2}\big/\binom n2 = \binom{q'}{3}\big/\binom n3 = 1$, so
$m_{a,\ell'}^{(n)} = \tfrac{a-2}{a}$ and
$b_{a,\ell'}^{(n)} = \tfrac{2(a-3)}{a}$, and the difference of the two sides of
\eqref{eq:halfplane} is
\begin{align*}
  (4-\ell^3)\Bigl(\frac{a-2}{a}-1\Bigr)
  &- (2-\ell^2)\Bigl(\frac{2(a-3)}{a}-2\Bigr) \\
  &\qquad =\; \frac{2\bigl(\ell^3-3\ell^2+2\bigr)}{a}
   \;=\; \frac{2\,(1-\ell)\bigl(3-(1-\ell)^2\bigr)}{a}.
\end{align*}
For $0 < \ell \leq 1$ both factors are nonnegative, since $(1-\ell)^2 \leq 1$.

\emph{Step 3: the endpoint $q' = q$.} Here both sides of \eqref{eq:halfplane} are
formed from the same $q$. Writing $\Delta$ for their difference and substituting
\eqref{eq:pure entries exact}, with $\ell = (q-2)/n$,
\begin{align*}
  \Delta
  &= (4-\ell^3)\left(\frac{a-2}{a}\cdot\frac{\binom q2}{\binom n2}-1\right)
   - (2-\ell^2)\left(\frac{2(a-3)}{a}\cdot\frac{\binom q3}{\binom n3}-2\right)
     \\[3pt]
  &= \frac{4n^3-(q-2)^3}{n^3}\cdot\frac{(a-2)q(q-1)-a\,n(n-1)}{a\,n(n-1)} \\[3pt]
  &\qquad{}- \frac{2n^2-(q-2)^2}{n^2}\cdot
     \frac{2\bigl[(a-3)q(q-1)(q-2)-a\,n(n-1)(n-2)\bigr]}{a\,n(n-1)(n-2)},
\end{align*}
each of the four factors having simply been written as a single fraction. Every
appearance of $a$ is now inside a factor $\tfrac{a-2}{a}=1-\tfrac2a$ or
$\tfrac{a-3}{a}=1-\tfrac3a$, so $\Delta$ is affine in $1/a$. Its two endpoint
values are
\begin{equation}\label{eq:endpoints}
  \Delta\big|_{a\to\infty} = \frac{(n-q)\,P}{n^4(n-1)(n-2)},
  \qquad
  \Delta\big|_{a=4} = \frac{N_0}{4\,n^4(n-1)(n-2)},
\end{equation}
with
\begin{align}
  P &= 2n\bigl[2n^2-(q-2)^2\bigr]\bigl[(q-2)(n+q-1)+n(n-1)\bigr]
       \notag \\
    &\qquad {}- (n-2)(n+q-1)\bigl[4n^3-(q-2)^3\bigr],
       \label{eq:P closed} \\[3pt]
  N_0 &= (n-2)\bigl[4n^3-(q-2)^3\bigr]\bigl[2q(q-1)-4n(n-1)\bigr]
       \notag \\
    &\qquad {}- 2n\bigl[2n^2-(q-2)^2\bigr]
        \bigl[q(q-1)(q-2)-4n(n-1)(n-2)\bigr].
       \label{eq:N0 closed}
\end{align}
Being affine in $1/a$, and with $1/a\in(0,\tfrac14]$ for $a\ge 4$, the value
$\Delta$ is a convex combination of these endpoints:
\begin{equation}\label{eq:convex comb}
  \Delta \;=\; \Bigl(1-\tfrac{4}{a}\Bigr)\,\Delta\big|_{a\to\infty}
             \;+\; \tfrac{4}{a}\,\Delta\big|_{a=4},
  \qquad a\ge 4,
\end{equation}
both weights nonnegative. It therefore suffices to show that the two endpoints
are nonnegative, that is, that $P\ge 0$ and $N_0\ge 0$ for $4\le q\le n$. We pass
to the variables
\[
  w = q-2 = L, \qquad M = n-q,
\]
so that $w\ge 2$ and $M\ge 0$ on this range. Under $q=w+2$, $n=w+M+2$ the six
polynomials from which $P$ is built in \eqref{eq:P closed} become short and
have nonnegative coefficients:
\begin{gather*}
  2n = 2M+2w+4, \qquad n-2 = M+w, \qquad n+q-1 = M+2w+3, \\
  2n^2-(q-2)^2 = 2M^2+4Mw+8M+w^2+8w+8, \\
  (q-2)(n+q-1)+n(n-1) = M^2+3Mw+3M+3w^2+6w+2, \\
  4n^3-(q-2)^3 = 4M^3+12M^2w+24M^2+12Mw^2 \\
  {}+48Mw+48M+3w^3+24w^2+48w+32.
\end{gather*}
Multiplying out the two products in \eqref{eq:P closed} and collecting by powers
of $w$ gives
\begin{multline}\label{eq:P grouped}
  P = 3(M+5)\,w^4 + (5M^2+57M+100)\,w^3 + 2(M+2)(M^2+23M+54)\,w^2 \\
    {}+ 12(M+2)^2(M+4)\,w + 8(M+2)^3,
\end{multline}
and the same substitution in \eqref{eq:N0 closed} gives
\begin{multline}\label{eq:N0 grouped}
  \tfrac14 N_0 = 3\,w^5 + 3(5M+7)\,w^4 + 2(15M^2+44M+25)\,w^3 \\
    {}+ (M+2)(23M^2+70M+24)\,w^2 + 2(M+2)^2(3M^2+13M+2)\,w + 4M(M+2)^3.
\end{multline}
Every coefficient of $w$ in \eqref{eq:P grouped} and \eqref{eq:N0 grouped} is a
polynomial in $M$ with nonnegative coefficients, so $P\ge 0$ and $N_0\ge 0$ for
$w,M\ge 0$, in particular on our range. By \eqref{eq:endpoints} both endpoints of
$\Delta$ are then nonnegative, and by the convex combination \eqref{eq:convex comb}
so is $\Delta$ itself, for every $a\ge 4$.
\end{proof}

\begin{remark}
The finite-$n$ computation cannot be replaced by a limit. As $n \to \infty$ with
$q = \ell n + O(1)$ and $\ell' = \ell$, the difference $\Delta$ of the two sides
approaches
\[
  a\bigl[(4-\ell^3)(m_{a,\ell}-1) - (2-\ell^2)(b_{a,\ell}-2)\bigr]
  \;=\; \ell^2(a-4)(\ell-1)^2(\ell+2),
\]
which vanishes at $a = 4$ and as $\ell \to 1$, the two values where the line is
pinned. Since $\ell \leq 1 - 2/n$, this limiting slack is $O(n^{-2})$, whereas
the exact entries \eqref{eq:pure entries exact} differ from their limits
\eqref{eq:pure entries} by $O(n^{-1})$; a limit argument thus loses the
inequality, and the endpoints \eqref{eq:endpoints} must be checked exactly.
\end{remark}

We can now prove the upper bound.

\begin{proof}[Proof of Theorem~\ref{thm:triangle upper bound}]
Let $G$ be $C_4$-free with clique complex $X(G)$ that is $2$-Leray over $\Bbbk$, and let
$B(G)$ be its normalized Betti diagram, as in Section~\ref{sec:resolutions}.

Let $\ell:= 1 - \kappa$. By \eqref{eq:min length} and 
Lemma \ref{lem:BS length} each summand $\pi_{a_i,L_i}$ in the (convex) 
BS-decomposition of $B(G)$ fulfills the conditions of Lemma \ref{lem:halfplane}. 
The density parameters $m$ and $b$ for $B(G)$, being convex combinations
of the density parameters of each $\pi_{a_i,L_i}$, 
then also fulfill:
\begin{equation}\label{eq:b upper}
  b \;\leq\; \frac{4-\ell^3}{2-\ell^2}\,(m - 1) + 2.
\end{equation}
Substituting 
\[ \ell = 1 - \kappa, \quad m = 1 - \varepsilon,  \quad \tau = b - 3m + 1 \]
and simplifying, gives:
\[
  \tau \;\leq\; \frac{(3-\kappa^2)\,\kappa\,\varepsilon}{1 + 2\kappa - \kappa^2},
\]
which is \eqref{eq:upper}.
\end{proof}

\section{The lower bound}\label{sec:lower}

In this section we prove Theorem~\ref{thm:triangle lower bound}. The argument is
elementary and uses no hypothesis beyond $C_4$-freeness. Throughout we use the
density parameter $b = 2t_0 + t_1$ of Section~\ref{sec:prelim}, and freely use the
identities \eqref{eq:normalized census} and \eqref{eq:t2 and t3}. All estimates
below are exact at every finite $n$; in particular we do not discard the
normalization factors $\tfrac{n}{n-2}$ that arise, since they approach $1$ from
the side that would invalidate the finite statements.

We bound the induced $2$-path density $t_2$ from below in two ways, and from above
using $C_4$-freeness.

\begin{lemma}\label{lem:t2 lower trivial}
For any graph $G = (V,E)$ we have $t_2 \geq 3m - 2b$.
\end{lemma}

\begin{proof}
By \eqref{eq:t2 and t3} we have $t_2 = 3m - 3t_0 - 2t_1$. Since $t_0 \geq 0$,
\[
  t_2 = 3m - 3t_0 - 2t_1 \;\geq\; 3m - 4t_0 - 2t_1 = 3m - 2b,
\]
using $b = 2t_0 + t_1$.
\end{proof}

\begin{lemma}\label{lem:t2 lower jensen}
For any graph $G = (V,E)$ on $n \geq 3$ vertices we have
\[
  t_2 \;\geq\; 3(m^2 + m - b) - \frac{3\varepsilon m}{n-2}.
\]
\end{lemma}

\begin{proof}
Applying Jensen's inequality to the convex function $x \mapsto \binom{x}{2}$ and
using $\sum_{v} \deg v = 2|E| = \varepsilon\, n(n-1)$, so that the average degree
is $\varepsilon(n-1)$, we obtain
\begin{align*}
  D = \frac{1}{\binom{n}{3}} \sum_{v \in V} \binom{\deg v}{2}
    &\;\geq\; \frac{n}{\binom{n}{3}} \binom{\varepsilon(n-1)}{2}
     \;=\; \frac{3\varepsilon\bigl(\varepsilon(n-1)-1\bigr)}{n-2} \\
    &\;=\; 3\varepsilon^{2} - \frac{3\varepsilon(1-\varepsilon)}{n-2}
     \;=\; 3\varepsilon^{2} - \frac{3\varepsilon m}{n-2}.
\end{align*}
Combining
$t_2 = D - 3t_3$ from \eqref{eq:t2 and t3} with $t_3 = b - 3m + 1$ and
$\varepsilon = 1-m$ gives
\[
  t_2 = D - 3t_3 \;\geq\; 3\bigl(\varepsilon^{2} - t_3\bigr)
        - \frac{3\varepsilon m}{n-2}
      \;=\; 3(m^{2} + m - b) - \frac{3\varepsilon m}{n-2}. \qedhere
\]
\end{proof}

\begin{remark}\label{rem:C5}
The correction term cannot be omitted. For $G = C_5$ one has
$\varepsilon = m = b = \tfrac12$ and $t_2 = \tfrac12$, so that
$3(m^2+m-b) = \tfrac34 > t_2$, while
$3(m^2+m-b) - \tfrac{3\varepsilon m}{n-2} = \tfrac34 - \tfrac14 = \tfrac12$
and Lemma~\ref{lem:t2 lower jensen} holds with equality.
\end{remark}

The next lemma is the only place where $C_4$-freeness enters.

\begin{lemma}\label{lem:t2 upper C4free}
For any $C_4$-free graph $G = (V,E)$ on $n \geq 3$ vertices we have
\[
  t_2 \;\leq\; \frac{3m\,\omega(G)}{n-2} \;=\; 3m\kappa\cdot\frac{n}{n-2}.
\]
\end{lemma}

\begin{proof}
Fix two non-adjacent vertices $u$ and $w$, and let $v_1, \dots, v_k$ be the
vertices for which $u v_i w$ is an induced $2$-path, that is, the common neighbors
of $u$ and $w$. If some $v_i, v_j$ were non-adjacent, then
$u v_i w v_j$ would be an induced $4$-cycle; since $G$ is $C_4$-free, the
$v_1, \dots, v_k$ form a clique, so $k \leq \omega(G) = \kappa n$. As every induced
$2$-path has a unique non-adjacent pair of endpoints, summing over the
$\binom{n}{2} - |E| = m\binom{n}{2}$ non-edges gives
\[
  |T_2| \;\leq\; \omega(G) \cdot m\binom{n}{2}
         \;=\; \frac{3m\,\omega(G)}{n-2}\binom{n}{3}.
\]
Note that the last
expression is \emph{larger} than $3m\kappa\binom n3$, by the factor
$\tfrac{n}{n-2}$; the two agree only in the limit.
\end{proof}

\begin{corollary}\label{cor:b lower}
For any $C_4$-free graph $G = (V,E)$ on $n \geq 3$ vertices we have
\[
  b \;\geq\; \max\left\{\,
      m^2 + m - \frac{m(\kappa n + \varepsilon)}{n-2}
      \, , \;\;
      \tfrac{3}{2}m\left(1 - \frac{\kappa n}{n-2}\right)
  \,\right\}.
\]
\end{corollary}

\begin{proof}
The two bounds come from pairing the upper estimate
$t_2 \leq \tfrac{3m\kappa n}{n-2}$ of Lemma~\ref{lem:t2 upper C4free} with each
lower estimate for $t_2$ and solving for $b$: Lemma~\ref{lem:t2 lower trivial}
gives the first, and Lemma~\ref{lem:t2 lower jensen} the second.
\end{proof}

\begin{proof}[Proof of Theorem~\ref{thm:triangle lower bound}]
By \eqref{eq:t2 and t3} we have $\tau = t_3 = b - 3m + 1$. Substituting the two
lower bounds of Corollary~\ref{cor:b lower},
\[
  \tau \;\geq\; \max\left\{\,
      m^2 - 2m + 1 - \frac{m(\kappa n+\varepsilon)}{n-2}
      \, , \;\;
      1 - \tfrac32m\left(1 + \frac{\kappa n}{n-2}\right)
  \,\right\}.
\]
Writing $m = 1-\varepsilon$, the first branch is
$m^2 - 2m + 1 - \tfrac{m(\kappa n+\varepsilon)}{n-2}
 = \varepsilon^2 - \tfrac{(1-\varepsilon)(\kappa n+\varepsilon)}{n-2}$
and the second is
$1 - \tfrac32(1-\varepsilon)\bigl(1+\tfrac{\kappa n}{n-2}\bigr)$, which is
\eqref{eq:lower}. Since $\tfrac{\kappa n+\varepsilon}{n-2} \to \kappa$ and
$\tfrac{\kappa n}{n-2} \to \kappa$ as $n \to \infty$, the right-hand side of
\eqref{eq:lower} converges to \eqref{eq:lower limit}.
\end{proof}

\section{Further discussion}\label{sec:discussion}

\paragraph{A clique-density bound by sandwiching}
Combining the two triangle-density bounds gives a lower bound on the
clique-number density in terms of $\varepsilon$ alone: a $C_4$-free graph must
satisfy both \eqref{eq:upper} and \eqref{eq:lower}, and eliminating $\tau$
between them yields \eqref{eq:sandwich low} and \eqref{eq:sandwich high}, using
the limiting form \eqref{eq:lower limit}. The resulting bound is therefore
asymptotic, as noted in the introduction. The two relations share the same
right-hand side $\tfrac{(3-\kappa^2)\kappa\varepsilon}{1+2\kappa-\kappa^2}$, so
which is binding is decided by their left-hand sides: \eqref{eq:sandwich low} is
active at small $\kappa$ and \eqref{eq:sandwich high} at large $\kappa$. 
Writing
$\ell = 1-\kappa$, the two left-hand sides are equal precisely when
$\varepsilon = 1 - \tfrac{\ell}{2}$.
Substituting this into either branch leaves
a single equation in $\ell$, the cubic $\ell^3 + 4\ell^2 - 14\ell + 8 = 0$, whose
unique root in $(0,1)$ is
\[
  \ell_0 = -\frac{4}{3} + \frac{2\sqrt{58}}{3}\,
    \cos\!\left( \tfrac{1}{3}\arccos\!\left( -\tfrac{106\sqrt{58}}{841} \right)
      - \tfrac{2\pi}{3} \right)
  \;\approx\; 0.7780.
\]
Solving each branch for $\kappa$ then gives the piecewise curve of
Figure~\ref{fig:bounds}: the purple portion ($\kappa \leq 1-\ell_0$) from
\eqref{eq:sandwich low}, the black portion ($\kappa \geq 1-\ell_0$) from
\eqref{eq:sandwich high}. This clique-density bound inherits the Leray hypothesis
of Theorem~\ref{thm:triangle upper bound}: it holds only for $C_4$-free graphs
whose clique complex is $2$-Leray over $\Bbbk$.

\paragraph{Comparison with known bounds}
It is instructive to compare the resulting bound with the two other curves drawn
in Figure~\ref{fig:bounds}. As $\varepsilon \to 0$, expanding the active branch
\eqref{eq:sandwich low} gives
\[
  \kappa \;=\; \varepsilon^2 - 2\varepsilon^3 + 4\varepsilon^4 - \cdots,
\]
so $\kappa \sim \varepsilon^2$ for sufficiently small $\varepsilon$. This has the same quadratic
order as the general $C_4$-free bound \eqref{eq:holmsen} of \cite{Holmsen2020}
(the red curve in Figure~\ref{fig:bounds}), but with leading coefficient $1$
rather than $\tfrac14$. The coefficient $1$ matches the refined small-$\varepsilon$
estimate of \cite{GyarfasHubenkoSolymosi2002}, while our bound holds across the
whole range. Across $\varepsilon \in (0,1)$ our bound (the purple and black curve
in Figure~\ref{fig:bounds}) lies strictly above the red curve and strictly below
the blue chordal curve $\kappa = 1-\sqrt{1-\varepsilon}$ of \eqref{eq:chordal};
thus it interpolates between the two, as one would expect of a bound for a graph
class lying between chordal and general $C_4$-free graphs. At the other extreme,
as $\varepsilon \to 1$ the active branch \eqref{eq:sandwich high} has the
expansion
\[
  \varepsilon \;=\; 1 - \tfrac12(\kappa-1)^2 + \tfrac18(\kappa-1)^3
    - \tfrac{3}{32}(\kappa-1)^4 + \cdots,
\]
whose leading terms agree with $\varepsilon = \tfrac12 + \kappa - \tfrac12\kappa^2$,
that is, with the curve $\kappa = 1-\sqrt{2}\sqrt{1-\varepsilon}$; this is the
asymptotic shape of our bound near $\varepsilon = 1$.

\paragraph{The point $\varepsilon = \tfrac12$}
The edge density $\varepsilon = \tfrac12$ is of special interest, since the
circulant graphs $G_k = C_{4k+1}^k$ of Example~\ref{ex:circulant graph} give the
best construction known there, with clique-number density approaching
$\tfrac14$. Our bound at this density is obtained by setting
$\varepsilon = \tfrac12$ in \eqref{eq:sandwich low}, which reduces to the cubic
\[
  4\ell^3 - 7\ell^2 - 4\ell + 6 = 0,
\]
whose unique root in $(0,1)$ is
\[
  \ell_1 = \frac{7}{12} + \frac{\sqrt{97}}{6}\,
    \cos\!\left( \tfrac{1}{3}\arccos\!\left( -\tfrac{449\sqrt{97}}{9409} \right)
      - \tfrac{2\pi}{3} \right)
  \;\approx\; 0.8499,
\]
giving $\kappa \geq 1 - \ell_1 \approx 0.1501$. There remains a gap between this
lower bound and the value $\tfrac14$ attained by the construction. While
Gy{\'a}rf{\'a}s and S{\'a}rk{\"o}zy \cite{GyarfasSarkozy2017} settled the regular
case, showing that $2k$-regular $C_4$-free graphs on $4k+1$ vertices cannot beat
$\kappa = \tfrac14$, the question over all $C_4$-free graphs is open. They also
show that more can be said under a hypothesis on the \emph{minimum} 
degree~\cite[Theorem 3]{GyarfasSarkozy2017}: a $C_4$-free graph with 
$\delta(G) \leq \tfrac{11}{15}n$
satisfies $\omega(G) \geq \delta(G) - \tfrac n3$. In particular, taking 
$\delta(G) \geq \tfrac n2$ yields $\omega(G) \geq \tfrac n6$, and they remark that whether 
the same conclusion follows from the edge-density hypothesis $\varepsilon \geq \tfrac 12$ remains open.

\paragraph{The chordal bound revisited}
Boij-S{\"o}derberg theory nearly immediately recovers the classical chordal bound 
\eqref{eq:chordal}, without appeal to perfectness, 
on which the proof of \cite[Theorem 3]{GyarfasHubenkoSolymosi2002} relied. 
A graph $G$ is chordal precisely when its clique complex is $1$-Leray, 
so by Hochster's formula the Betti diagram $B(G)$ then has a \emph{single} linear strand: 
its only nonzero rows are $d = j-i \in \{0,1\}$. By the classification of two-strand types
the one-strand pure diagrams are exactly those of type $\mathbf{d}(L+2,L) = (0,2,3,\dots,L+1)$, the
boundary case $a = L+2$ of the family $\pi_{a,L}$. By \eqref{eq:strand entries}
their first strand entry is $\hat\beta_{1,2} = \binom{L+1}{2}$, so such a diagram
of length $\ell n$ has density parameter $m \to \ell^2$.

Write the Boij--S\"oderberg decomposition $B(G) = \sum_i \alpha_i \pi_{L_i+2,\,L_i}$.
Since density is additive over the decomposition, $m = \sum_i \alpha_i \ell_i^2$,
where $\ell_i$ is the length of the $i$th summand. By \eqref{eq:min length} each
$\ell_i \geq 1 - \kappa$, and as the weights $\alpha_i$ are convex the sum is
minimized when every $\ell_i = 1-\kappa$, that is, when the decomposition is a
single one-strand diagram of length $(1-\kappa)n$. Any longer summand would only
increase $m = \hat\beta_{1,2}$, the missing-edge density, and so decrease the
number of edges of $G$. Thus, for a chordal graph of clique-number density
$\kappa$,
\[
  m \;\geq\; (1-\kappa)^2,
  \qquad\text{equivalently}\qquad
  \kappa \;\geq\; 1 - \sqrt{1 - \varepsilon},
\]
with equality for the graph realizing the single diagram $\pi_{L+2,L}$, where
$1 - \kappa = \ell = L/n$. This is the chordal
bound \eqref{eq:chordal}, and the argument shows it is attained, recovering the
result of \cite{AbbottKatchalski1979, GyarfasHubenkoSolymosi2002}. The extremal
graph is the complete split graph $\overline{K_{L+1}}\ast K_{n-L-1}$, the join of an
independent set of size $L+1$ with a clique on the remaining vertices:
it is chordal, has clique number $n - L$, and its only missing edges are the
$\binom{L+1}{2}$ within the independent set, so $m = \binom{L+1}{2}/\binom{n}{2} \to
(1-\kappa)^2$. Its single one-strand diagram is precisely that of the edge ideal
of $K_{L+1}$, which has a linear resolution.

\paragraph{Forbidding longer holes}
The chordal bound is the extreme case of a more general estimate. Forbidding only
the shortest holes --- those of length up to some $g$, rather than all holes at
once --- still bounds the clique density from below, by a curve of the same shape
that approaches the $C_4$-free case as $g$ decreases.

\begin{prop}\label{prop:holes}
Let $g \geq 4$, and let $G$ be a graph on $n \geq 3$ vertices with at least one
edge, whose clique complex is $2$-Leray over $\Bbbk$ and which has no hole of
length at most $g$. Then
\begin{equation}\label{eq:hole exact}
  m \;\geq\; \frac{g-2}{g}\cdot
    \frac{\binom{n-\omega(G)+2}{2}}{\binom{n}{2}}
  \geq  \frac{g-2}{g}\,(1-\kappa)^2 ,
\end{equation}
with the middle term approaching the last term as $n \to \infty$.
Consequently
\begin{equation}\label{eq:hole bound}
  \kappa \;\geq\; 1 - \sqrt{\frac{g}{g-2}}\;\sqrt{1-\varepsilon}.
\end{equation}
\end{prop}

\begin{proof}
Write $B(G) = \sum_i \alpha_i \pi(\mathbf{d}^{(i)})$ as in \eqref{eq:BS}. Each
type has at most two linear strands and, $G$ being $C_4$-free, is of the form
$\mathbf{d}(a_i, L_i)$. Its row-$2$ strand opens in internal degree $a_i+1$;
since $G$ has no hole of length at most $g$, the last restriction of
Section~\ref{sec:resolutions} forces $d_p - p \leq 1$ for all $d_p \leq g$, so
$a_i + 1 > g$, that is $a_i \geq g$, for every $i$. Since $G$ has an edge,
$\omega(G) \geq 2$, so $L_i \geq n - \omega(G)$ by \eqref{eq:min length} while
$L_i \leq n-2$; in particular $n-\omega(G) \leq n-2$.

The entry $\beta_{1,2}$ is additive over the Boij--S\"oderberg decomposition, so
by \eqref{eq:strand entries}
\[
  m\binom{n}{2} \;=\; \sum_i \alpha_i\,\frac{a_i-2}{a_i}\binom{L_i+2}{2}.
\]
Both $\tfrac{a-2}{a}$ and $\binom{L+2}{2}$ increase in their arguments, so each
summand is at least $\tfrac{g-2}{g}\binom{n-\omega(G)+2}{2}$; as
$\sum_i \alpha_i = 1$, the same bound holds for the sum, giving the first
inequality of \eqref{eq:hole exact}. 
Furthermore
\[ \frac{\binom{n-\omega +2}{2}}{\binom{n}{2}} = 
\frac{(n(1-\kappa) +2)((n-1)(1-\kappa) + 2-\kappa)}{n(n-1)} 
= (1-\kappa)^2 + \text{ positive terms } \]
giving the second inequality. 
The limit statement is immediate, 
and \eqref{eq:hole bound}
follows on setting $m = 1-\varepsilon$ and solving for $\kappa$.
\end{proof}

As $g$ ranges from $4$ to $\infty$ the bound \eqref{eq:hole bound} interpolates
between the two curves already in play. At $g = 4$ it asks nothing beyond
$C_4$-freeness and gives $\kappa \geq 1-\sqrt{2}\,\sqrt{1-\varepsilon}$, the
branch to which the bound of this section is asymptotic as $\varepsilon \to 1$;
as $g \to \infty$ it returns the chordal bound \eqref{eq:chordal}, the hole-free
case.

How close is \eqref{eq:hole bound} to the truth? The circulant graphs of
Example~\ref{ex:circulant graph} generalize to a family that tests it.
The circulant graph $C^k_{gk+1}$ is $2$-Leray and has no holes in
the range $[4,g]$. The join $C^k_{gk+1} \ast K_r$ has vertex number, clique number, and
edge number respectively:

\[ (i)\,\, n = gk+1+r, \quad 
(ii)\, k+1+r = n-(g-1)k, \quad (iii)\, \, \binom{n}{2} - \frac{(gk+1)(g-2)k}{2}. \]
So the clique density is $\kappa = 1 - (g-1) \frac{k}{n}$.
From (\emph{i}) one works out 
\[ \frac{k}{n}  = \frac{1}{g} - \frac{1+r}{gn} = \frac{1}{g} - \frac{1+r}{g(gk+r+1)}.\]
As $k$ gets larger, and varying $r$, this becomes more and more dense in the interval 
$(0,\frac{1}{g})$. The clique density $\kappa$ above correspondingly becomes dense in the interval $(\frac{1}{g},1)$. 

Concerning the edge density, by (\emph{iii}) it is
\[ \varepsilon = 1 - (g^2-2g) (\frac{k}{n})^2 + O(\frac{1}{n}).\]
This gives 
\[ \frac{(1-\kappa)^2}{1 -\varepsilon} = \frac{(g-1)^2}{g^2-2g} + O(\frac{1}{n}).\]
Furthermore, as $\kappa$ ranges over $(\frac{1}{g},1)$, one sees that $\varepsilon$ ranges over
$(\frac{2}{g},1)$. Solving for $\kappa$ this gives 
\[ \kappa = 1 - \sqrt{\frac{(g-1)^2}{g^2-2g}} \sqrt{1-\varepsilon} + O(\frac{1}{n}).\]

Comparing the bound \eqref{eq:hole bound}, and circulant graphs as $n \to \infty$ one has:
\[ \text{Bound: }   \kappa \;\geq\; 1 - \sqrt{\frac{g^2}{g^2-2g}}\;\sqrt{1-\varepsilon},
\quad \text{Circulant joins: } 
\kappa = 1 - \sqrt{\frac{(g-1)^2}{g^2-2g}} \sqrt{1-\varepsilon}.\]

\begin{conj*} Let $g \geq 3$ be an integer. For $2$-Leray graphs over a field $\Bbbk$
with no holes in the range $[4,g]$, if the edge density $\varepsilon \in (\frac{2}{g}, 1)$,
the clique density is asymptotically (as $n \to \infty$) bounded by 
\[ \kappa \geq 1 - \sqrt{\frac{(g-1)^2}{g^2-2g}} \sqrt{1-\varepsilon}.\]
\end{conj*}

\paragraph{The Leray hypothesis and open questions}
The lower bound of Theorem~\ref{thm:triangle lower bound} is unconditional. The
upper bound of Theorem~\ref{thm:triangle upper bound}, and hence the
clique-density bound above, is proved only under the assumption that the clique
complex is $2$-Leray over $\Bbbk$. No $C_4$-free graph is known to us that
violates \eqref{eq:upper}, and we do not know whether the hypothesis can be
dropped. This leads to Question \ref{qu:tau bound} in the introduction.

\medskip
More broadly, \eqref{eq:upper} is one instance of a single extremal problem for
the triangle density: determine the optimal function $f$ with
\[
  \tau \;\leq\; f(\varepsilon,\kappa)
\]
for every $C_4$-free graph, the largest triangle density attainable at prescribed
edge and clique-number densities. Theorem~\ref{thm:triangle upper bound} supplies
an upper bound for $f$ on the $2$-Leray subclass; the question above is whether that 
same bound holds for $f$ on
the whole class.

Some evidence for this comes from graphs whose clique complexes carry homology
in dimension $2$ or higher, and so fall outside the scope of
Theorem~\ref{thm:triangle upper bound}. The $1$-skeleton of the icosahedron is
$C_4$-free, and its clique complex is a flag triangulation of $S^2$, so
$\tilde H_2 \neq 0$ and the complex is not $2$-Leray. Here $\kappa = \tfrac14$,
$\varepsilon = \tfrac{5}{11}$, and $\tau = \tfrac{1}{11}$; inequality
\eqref{eq:upper} holds comfortably, with slack $\tfrac{13}{64}$ between its two
sides in the reduced form
$\tau(1+2\kappa-\kappa^2) \leq (3-\kappa^2)\kappa\varepsilon$. Likewise the
$1$-skeleton of the $600$-cell is $C_4$-free, with clique complex a flag
triangulation of $S^3$, so $\tilde H_3 \neq 0$; here $\kappa = \tfrac{1}{30}$,
$\varepsilon = \tfrac{12}{119}$, and $\tau = \tfrac{30}{7021}$, and
\eqref{eq:upper} again holds, now with the much smaller slack
$\tfrac{43658}{7898625} \approx 5.5\times 10^{-3}$. Thus the bound persists on
these two flag spheres even though the homological hypothesis fails.

Moreover, neither example is isolated: each seeds an infinite supply of
further examples, produced by a single operation that preserves
$C_4$-freeness, non-$2$-Lerayness, and \eqref{eq:upper} together. Let $C$
be a clique of $G$ of size $k$, and let $s\ge k$. Write $G\oplus_C K_s$ for
the clique-sum of $G$ with $K_s$ along $C$: one adjoins $s-k$ mutually
adjacent new vertices, each joined to every vertex of $C$ and to nothing
else. Taking $k=3$, $s=4$ is the case of stellar subdivision of a triangle.

\begin{theorem}\label{thm:cliquesum}
Let $G$ be a $C_4$-free graph on $n\ge 3$ vertices, $C$ a clique of $G$
with $|C|=k$, and $s\ge k$. Then $G\oplus_C K_s$ is $C_4$-free, contains $G$
as an induced subgraph---so if $X(G)$ is not $2$-Leray then neither is
$X(G\oplus_C K_s)$---and, if $G$ satisfies \eqref{eq:upper}, then so does
$G\oplus_C K_s$.
\end{theorem}

\begin{proof}[Sketch]
It suffices to treat $s=k+1$, a single cone $Q_C(G)$ over $C$: gluing
$K_s$ is the same as coning $s-k$ times, each new vertex placed on the
clique formed by $C$ and the vertices already added, so the general case
follows by iterating this one.

Write $v$ for the coned vertex. It is adjacent to the clique $C$, so any
two of its neighbours are adjacent; as the two vertices opposite across $v$
on an induced $4$-cycle would be nonadjacent, $v$ lies on no such cycle,
and a $4$-cycle among the old vertices would already lie in the $C_4$-free
graph $G$. Since the cone adds no edges within $V(G)$, the subcomplex of
$X(Q_C(G))$ induced on $V(G)$ is $X(G)$, and $2$-Lerayness is inherited by
induced subcomplexes, giving the first two claims. For the third, write
$N(n,e,t,\omega)$ for the numerator of \eqref{eq:upper} cleared of
denominators, so the hypothesis is $N\ge 0$. The cone over a $k$-clique
sends
\[
  (n,e,t,\omega)\;\longmapsto\;
  \Bigl(n+1,\;e+k,\;t+\tbinom{k}{2},\;\max(\omega,k+1)\Bigr).
\]
One then checks a polynomial identity
$D\cdot N(Q_C(G)) = \lambda\,N(G) + \nu\bigl(e-\tbinom{\omega}{2}\bigr) + R$
with $D=3n(n^2+2n\omega-\omega^2)>0$ and $\lambda,\nu,R$ nonnegative on the
range $1\le k\le\omega<n$, in two cases according to whether $k\le\omega-1$
(the clique number is unchanged) or $k=\omega$ (it rises by one). Since
$G$ contains $K_\omega$, we have $e\ge\binom{\omega}{2}$, and the identity
gives $N(Q_C(G))\ge 0$.
\end{proof}

Starting from the icosahedron and coning repeatedly---at each step over any
clique of the current graph---produces $C_4$-free graphs on arbitrarily
many vertices, none of them $2$-Leray, and all satisfying \eqref{eq:upper}; each
retains the icosahedron as an induced subcomplex, so $\tilde H_2\neq 0$
throughout. Starting from the $600$-cell gives graphs with
$\tilde H_3\neq 0$ instead. The clique coned at each step is free to be
chosen, so these are not single families but branching supplies of examples
of unbounded size. In particular the bound is not an accident of the two
smallest flag spheres.

These families stay within homological dimensions~$2$ and~$3$,
but the phenomenon reaches higher. For every $d$ there exist
$C_4$-free graphs whose clique complexes have nonvanishing reduced homology in
dimension $d$. Indeed, a clique complex
has no induced $4$-cycle in its $1$-skeleton precisely when it is
\emph{flag-no-square} (equivalently $5$-large) in the sense of Januszkiewicz and
Swi\k{a}tkowski. They constructed flag-no-square complexes with nonvanishing
homology in every dimension, in order to produce hyperbolic Coxeter groups of
arbitrarily large virtual cohomological
dimension~\cite{JanuszkiewiczSwiatkowski2003}. Osajda~\cite{Osajda2013} later gave
a simpler recursive construction, in which the dimension of nonvanishing homology
increases by one at each step. These constructions are indirect, however:
they produce no explicit graph on which \eqref{eq:upper} could be tested, and any
such graph is necessarily large. The vertex count grows at least exponentially in
the homological dimension, as the following refinement of a bound of Adiprasito,
Nevo, and Tancer records (exponential growth also follows from the logarithmic
regularity estimates of Dao, Huneke, and Schweig~\cite{DaoHunekeSchweig2013}).

\begin{lemma}\label{lem:homology vertices}
For $d \geq 1$ let $f(d)$ be the least number of vertices of a $C_4$-free graph
$G$ whose clique complex $X(G)$ has nonvanishing reduced homology
$\tilde H_d(X(G); \Bbbk)$. Then
\[
  f(d) \;\geq\; 2^{d+1} + d .
\]
\end{lemma}

\begin{proof}
For a vertex $v$ of $G$ write $N_v$ for its set of neighbors; in $X(G)$ this is
the vertex set of the link of $v$. Adiprasito, Nevo, and Tancer show, in the
course of proving \cite[Theorem 5.1]{AdiprasitoNevoTancer2020}, that if $X(G)$ has
nonvanishing reduced homology in dimension $d$ then $G$ has a vertex $v$ with
$|V(G) \setminus (N_v \cup \{v\})| \geq 2^d$; that is, some vertex has at least
$2^d$ non-neighbors besides itself.

Let $G$ realize $f(d)$, and let $v$ be such a vertex. Applying the Mayer--Vietoris
sequence to the cover of $X(G)$ by the deletion $X(G) \setminus \{v\}$ and the
closed star of $v$, whose intersection is the link of $v$, propagates nonvanishing
homology in dimension $d$ to a subcomplex on the neighbors of $v$ carrying
homology in dimension $d-1$; that subcomplex is again $C_4$-free, so it has at
least $f(d-1)$ vertices. Counting $v$, its $\geq 2^d$ non-neighbors, and the
vertices supporting the lower-dimensional homology yields
\[
  f(d) \;\geq\; f(d-1) + 2^d + 1 .
\]
Since $f(1) = 5$ (the smallest $C_4$-free graph with a nonzero first homology of
its clique complex is the pentagon $C_5$), induction gives
$f(d) \geq 5 + \sum_{i=2}^{d}(2^i + 1) = 2^{d+1} + d$.
\end{proof}

Theorem~\ref{thm:cliquesum} glues on a complete graph. We do not know
whether \eqref{eq:upper} is preserved under the clique-sum of $G$ with an
arbitrary $C_4$-free graph, nor under the join $G\ast K_s$ that adjoins $s$
universal vertices.

A further open problem is whether the sandwich bound is tight. The circulant
construction at $\varepsilon = \tfrac12$ leaves a gap, and it would be
interesting to determine the true extremal clique-number density for $C_4$-free
graphs, either at $\varepsilon = \tfrac12$ or across the whole range.
The method of this paper reduces such questions to understanding which Betti
tables, subject to the edge ideal and Leray constraints, actually arise from
clique complexes of $C_4$-free graphs.

\subsection*{How far the strand method reaches}

One might hope that relaxing the $2$-Leray hypothesis of
Theorem~\ref{thm:triangle upper bound} one homological dimension at a time yields
a corresponding family of bounds. The upper bound was obtained by confining
$B(G)$ to two linear strands and bounding, over all admissible pure diagrams, the
density pair $(m,b) = \bigl(\beta_{1,2}/\binom n2,\ \beta_{2,3}/\binom n3\bigr)$ by
a line through the point $(1,2)$ and the extremal two-strand diagram; the
triangle density then followed from $\tau = b - 3m + 1$. The same construction
makes sense for any number $k$ of strands. Allowing rows $d = j-i \in \{0,1,\dots,k\}$
admits diagrams whose second row enters late, and Lemma~\ref{lem:homology vertices}
bounds how late: a nonzero entry in row $r$ cannot occur before internal degree
$2^{r}+r-2$. The extremal $k$-strand diagram---the one minimal in the
Boij--S\"oderberg order whose higher strands open as early as this permits---then
occupies a definite point
\[
  P^\ast_k \;=\; \bigl(\gamma_k\,\ell^2,\ \delta_k\,\ell^3\bigr),
  \qquad
  \gamma_k = \prod_t \frac{a_t-2}{a_t},
  \quad
  \delta_k = 2\prod_t \frac{a_t-3}{a_t},
\]
where $\ell = 1-\kappa$ and the $a_t$ are the internal degrees at which successive
strands open, namely $a_1 = 4$ and $a_r = 2^{r}+r-3$ for $r \geq 3$. The line
through $(1,2)$ and $P^\ast_k$ is the candidate $k$-strand estimate; for $k=2$ it
is the line of Lemma~\ref{lem:halfplane}, with $P^\ast_2 = (\tfrac12\ell^2,
\tfrac12\ell^3)$.

For $k=3$ this candidate is in fact a theorem. The extremal point is
$P^\ast_3 = (\tfrac38\ell^2, \tfrac{5}{16}\ell^3)$, and the line through it and
$(1,2)$ dominates every admissible three-strand pure diagram at the exact entries;
the proof is the same reduction to two endpoints as
Lemma~\ref{lem:halfplane}, the endpoint at the shorter length again reducing to a
polynomial with nonnegative coefficients after the substitution $q = 8+B+u$,
$n = q+s$ that encodes the admissible ranges $a_1 \geq 4$, $a_2 \geq 8$. The
resulting bound,
\[
  \tau \;\leq\;
  \frac{\varepsilon\,(5\kappa^3 + 3\kappa^2 - 21\kappa - 3)}
       {2\,(3\kappa^2 - 6\kappa - 5)},
\]
holds for every $C_4$-free graph whose clique complex is $3$-Leray. It is weaker
than \eqref{eq:upper}, as it must be, since it applies to the larger class;
combined with Theorem~\ref{thm:triangle lower bound} it gives a clique-density
curve lying below the one of Figure~\ref{fig:bounds}. We have not sought exact
certificates for $k \geq 4$, where the endpoint polynomials grow rapidly and the
requisite positivity is no longer a routine check.

Even for $k \geq 4$, where we do not prove that the line is a valid bound, we can
still say exactly where that line would sit. Its two defining points are $(1,2)$
and the extremal point $P^\ast_k$, and $P^\ast_k$ is fixed by
Lemma~\ref{lem:homology vertices} alone---its coordinates depend only on the
degrees $a_t$ at which the strands open, not on any positivity certificate. As $k$
grows these points converge:
\[
  \gamma_k \longrightarrow \gamma_\infty = 0.29298\ldots,
  \qquad
  \delta_k \longrightarrow \delta_\infty = 0.21402\ldots
  \qquad (k \to \infty).
\]
Consequently the lines through $(1,2)$ and $P^\ast_k$ approach a definite limiting
line, and no bound obtainable from this construction---for any $k$, and even
granting every certificate one might hope to prove---can lie above the
clique-density curve it produces. That limiting curve is drawn in
Figure~\ref{fig:limit}. It meets the axis $\kappa = 0$ at
$\varepsilon_\infty = 0.4873\ldots$, the increasing sequence of intercepts
$\tfrac13, \tfrac{5}{12}, \dots$ of the successive curves converging to this
value; below $\varepsilon_\infty$ the construction says nothing at all. In
particular it gives no information at the extremal edge density
$\varepsilon = \tfrac12$, where the circulant construction of
Example~\ref{ex:circulant graph} leaves the principal gap. The limiting curve is
not, however, uniformly weaker than the general bound \eqref{eq:holmsen}: the two
cross near $\varepsilon = 0.77$, and for larger edge densities the limiting curve
lies above \eqref{eq:holmsen}, by a margin rising to about $0.019$ in $\kappa$
near $\varepsilon = 0.93$. Thus even pushed to its limit the construction would
improve on the previously known bound, but only in the high-density range, and
never near $\varepsilon = \tfrac12$.

We stress that this ceiling is a limitation of \emph{the particular
construction} of this paper, not of Boij--S\"oderberg theory applied to the problem. The line
lives in the $(m,b)$-plane spanned by the two first-strand entries $\beta_{1,2}$
and $\beta_{2,3}$; a diagram with three or more strands carries further entries,
such as a third-strand value $\beta_{3,\ast}$, that this projection discards. A
sharper argument might retain such an entry and work in an enlarged
$(m,b,t)$-space, where constraining $t$ could exclude points that the planar
projection cannot, or might exploit that not every diagram admissible under
Lemma~\ref{lem:homology vertices} arises from an actual clique complex. We do not
know whether either route helps. What the limiting curve does show is that the
present idea, refined only by admitting more strands, cannot reach
$\varepsilon = \tfrac12$: progress there requires an ingredient outside this
family.

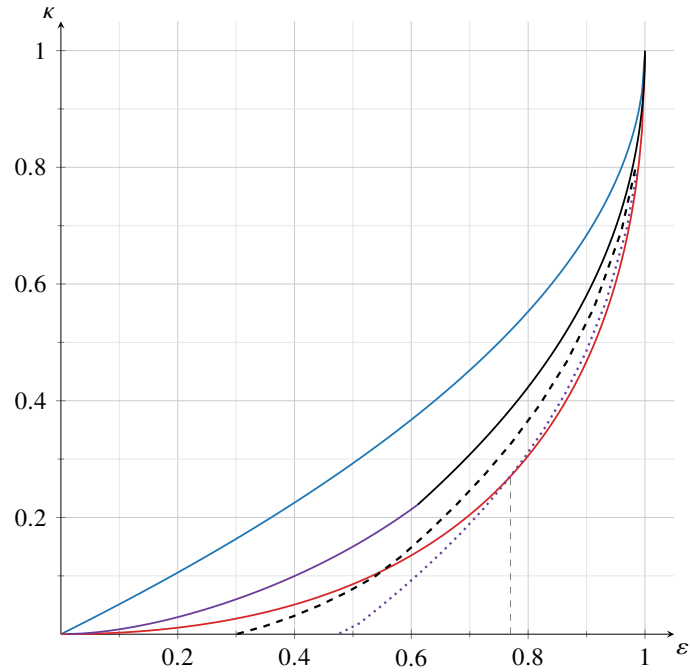
\begin{figure}[htbp]
    \centering
\begin{tikzpicture}[scale=0.85]
\begin{axis}[
    width=0.8\linewidth, height=0.8\linewidth,
    xmin=0, xmax=1.05, ymin=0, ymax=1.05,
    axis lines=middle,
    xlabel={$\varepsilon$},
    ylabel={$\kappa$},
    xlabel style={at={(axis description cs:1.04, 0.0)},anchor=north east},
    ylabel style={at={(axis description cs:-0.02,0.99)},anchor=south,rotate=0},
    xtick={0,0.2,0.4,0.6,0.8,1.0},
    ytick={0,0.2,0.4,0.6,0.8,1.0},
    minor tick num=1,
    grid=both,
    major grid style={line width=.3pt,draw=gray!40},
    minor grid style={line width=.15pt,draw=gray!20},
    tick label style={font=\small},
    label style={font=\small},
    samples=200,
    clip=true,
]
\addplot[cChordal, thick, domain=0:1] {1 - sqrt(1-x)};
\addplot[cHolmsen, thick, domain=0:1] {(1 - sqrt(1-x))^2};
\addplot[cLow, thick, samples=200, domain=0.0001:0.221992, variable=\k]
    ( { (sqrt(\k)*sqrt(\k^4 - 3*\k^3 + 5*\k + 1) - \k*(\k-1)) / (1 + 2*\k - \k*\k) } , \k );
\addplot[cHigh, thick, samples=200, domain=0.221992:0.9999, variable=\k]
    ( { (-3*\k^3 + 5*\k^2 + 5*\k + 1) / (-\k^3 + 3*\k^2 + 3*\k + 3) } , \k );
\addplot[cHigh, line width=1pt, dashed] coordinates {
(0.3020,0.0005) (0.3490,0.0137) (0.3960,0.0300) (0.4429,0.0494) (0.4899,0.0720)
(0.5369,0.0982) (0.5839,0.1343) (0.6309,0.1768) (0.6779,0.2226) (0.7248,0.2726)
(0.7718,0.3282) (0.8188,0.3916) (0.8658,0.4664) (0.9128,0.5597) (0.9598,0.6916)
(0.9833,0.7965)};
\addplot[cLow, line width=1pt, dotted] coordinates {
(0.4770,0.0013) (0.5121,0.0192) (0.5472,0.0473) (0.5824,0.0768) (0.6175,0.1079)
(0.6526,0.1409) (0.6877,0.1762) (0.7228,0.2142) (0.7579,0.2557) (0.7931,0.3015)
(0.8282,0.3529) (0.8633,0.4120) (0.8984,0.4821) (0.9335,0.5702) (0.9687,0.6947)
(0.9862,0.7926)};
\addplot[gray, thin, dashed] coordinates {(0.7698,0) (0.7698,0.2706)};
\end{axis}
\end{tikzpicture}
    \caption{The horizon of the strand construction. Solid: the chordal bound
    (blue), the general $C_4$-free bound \eqref{eq:holmsen} of Holmsen (red), and
    the proven bound of Figure~\ref{fig:bounds} (purple/black, here the $k=2$
    case). Dashed: the proven $3$-Leray bound ($k=3$). Dotted: the limiting curve
    of the construction as $k \to \infty$. The limiting curve gives no
    information left of $\varepsilon_\infty = 0.487\ldots$; it crosses \eqref{eq:holmsen} near
    $\varepsilon = 0.77$ (grey line) and improves on it only for larger
    $\varepsilon$.}
    \label{fig:limit}
\end{figure}

\end{document}